\documentclass{article}

\usepackage{arxiv}

\usepackage[utf8]{inputenc} % allow utf-8 input
\usepackage[T1]{fontenc}    % use 8-bit T1 fonts
\usepackage{hyperref}       % hyperlinks
\usepackage{url}            % simple URL typesetting
\usepackage{booktabs}       % professional-quality tables
\usepackage{amsfonts}       % blackboard math symbols
\usepackage{nicefrac}       % compact symbols for 1/2, etc.
\usepackage{microtype}      % microtypography
\usepackage{lipsum}         % Can be removed after putting your text content
\usepackage{graphicx}
\usepackage[authoryear]{natbib}
\usepackage{doi}

\usepackage{multirow}%
\usepackage{amsmath,amssymb}%
\usepackage{amsthm}%
\usepackage{mathrsfs}%
\usepackage[title]{appendix}%
\usepackage{xcolor}%
\usepackage{textcomp}%
\usepackage{manyfoot}%
\usepackage{algorithm}%
\usepackage{algorithmicx}%
\usepackage{algpseudocode}%
\usepackage{listings}%
\usepackage{longtable}%
\usepackage{cleveref}       % smart cross-referencing

\newtheorem{theorem}{Theorem}%  meant for continuous numbers
\newtheorem{corollary}{Corollary}%

\title{Exploring the Zipf Distribution Through the Lens of Mixtures}

\newif\ifuniqueAffiliation
\uniqueAffiliationtrue

\ifuniqueAffiliation % Standard variant of author block
\author{ \href{https://orcid.org/0000-0003-3675-6902}{\includegraphics[scale=0.06]{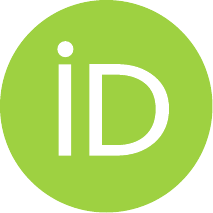}\hspace{1mm}Marta P\'erez-Casany}\thanks{Corresponding author} \\
	Department of Statistics and O. R.\\
	Technical University of Catalonia\\
	\texttt{marta.perez@upc.edu} \\
	\And
	\href{https://orcid.org/0000-0002-7432-0344}{\includegraphics[scale=0.06]{orcid.pdf}\hspace{1mm}Ariel Duarte-L\'opez} \\
	Department of Statistics and O. R.\\
	Technical University of Catalonia\\
	\texttt{ariel.duarte.lopez@upc.edu} \\
    \And
    \href{https://orcid.org/0000-0002-7827-0225}{\includegraphics[scale=0.06]{orcid.pdf}\hspace{1mm}Jordi Valero} \\
	Department of Statistics and O. R.\\
	Technical University of Catalonia\\
	\texttt{jordi.valero@upc.edu} \\
}
\else
\usepackage{authblk}

\newbox{\orcid}\sbox{\orcid}{\includegraphics[scale=0.06]{orcid.pdf}} 
\author[1]{%
	\href{https://orcid.org/0000-0000-0000-0000}{\usebox{\orcid}\hspace{1mm}David S.~Hippocampus\thanks{\texttt{hippo@cs.cranberry-lemon.edu}}}%
}
\author[1,2]{%
	\href{https://orcid.org/0000-0000-0000-0000}{\usebox{\orcid}\hspace{1mm}Elias D.~Striatum\thanks{\texttt{stariate@ee.mount-sheikh.edu}}}%
}
\affil[1]{Department of Computer Science, Cranberry-Lemon University, Pittsburgh, PA 15213}
\affil[2]{Department of Electrical Engineering, Mount-Sheikh University, Santa Narimana, Levand}
\fi

\renewcommand{\shorttitle}{\textit{arXiv} Template}

\hypersetup{
pdftitle={A template for the arxiv style},
pdfsubject={q-bio.NC, q-bio.QM},
pdfauthor={David S.~Hippocampus, Elias D.~Striatum},
pdfkeywords={First keyword, Second keyword, More},
}

\begin{document}
\maketitle

\begin{abstract}
	The Zipf distribution is a probability distribution widely used by scientists from various disciplines due to its ubiquity. Some of these areas include linguistics, physics, genetics, and sociology, among others. In this paper, it is proved that the Zipf distribution is both a mixture of geometric distributions and a mixture of zero-truncated Poisson distributions. It is also shown that it is not the zero-truncation of a mixed Poisson distribution. These results are important because they provide insights on the data generation mechanism that leads to data from a Zipf distribution. Additionally, it is proved, as a corollary, that the Zipf-Poisson Stopped Sum distribution is a particular case of a mixed Poisson distribution. The results are illustrated analyzing the 135 chapters of the novel Moby Dick.
\end{abstract}

% keywords can be removed
\keywords{Mixed distributions \and Zipf-Poisson Stopped Sum \and Zero-truncated Poisson distribution}

\section{Introduction}
Mixed Poisson (MP) distributions are the natural alternatives to the Poisson distribution when overdispersion is present. This is a consequence of considering that the Poisson parameter instead of being fixed is a random variable (r.v) allows to increase the variance of the final distribution. The Zipf distribution defined in $\{1,2,3,\cdots\}$, also known as discrete Pareto distribution, has proved to be very useful to fit different phenomena especially in the data tail. Some examples are, for instance, the World Wide Web, metabolic or protein  networks, city-size distribution on a country, frequencies of words in a text, etc. It is its simplicity and  ability to describe data from seemingly unrelated situations what increases its use among practitioners. 
In this paper it is proved that the Zipf may be written as a mixture of geometric distributions. As a consequence of this, it is also proved that it is also  a mixture of zero-truncated Poisson distributions (MZTP). In both cases the mixing distribution is analytically specified in the corresponding theorems. One could also think that the Zipf may also be obtained as the zero-truncation of a MP Distribution (ZTMP), and here it is proved  that this is not the case. As a consequence of the results presented, it can be deduced that the Zipf-Poisson Stopped Sum distribution (Zipf-PSS) \cite{duarte2020zipf} is a MP distribution.

The rest of the paper is organized as follows: Section \ref{section_zipf} is devoted to introduce the Zipf distribution, in Section \ref{sec_zipfContrib} the theoretical results are proved. Section \ref{section_poc} is devoted to a proof of concept which analyses the 135 chapters of the novel Moby Dick. Finally, Section \ref{section_conclusions} contains the main conclusions.

\section{The Zipf distribution}
\label{section_zipf}
The Zipf distribution  is a one-parameter distribution defined on the strictly positive integer numbers, where the probabilities change inversely to a power of the values. It is said that a r.v. $X$ follows a Zipf distribution with parameter $\alpha > 1$ if, and only if, its probability mass function (PMF) is equal to:
\begin{align}
\label{eqn_zipf_pmf}
P(X=x) = \frac{x^{-\alpha}}{\zeta(\alpha)},\,\, x = 1,2,...,\,\, \alpha >1,
\end{align}
where $\zeta(\alpha) = \sum_{i=1}^{+\infty} i^{-\alpha}$ is the Riemann Zeta function. Observe that the parameter space of the Zipf distribution is the set of values where the Riemann zeta function converges, which is $(1, +\infty)$.

Since it is markedly a skewed distribution, one may observe in a sample from this model values that sometimes differ by orders of magnitude. It is highly recommended for modeling two types of data: rank and frequencies of frequency data.
For frequencies of frequency data, one understands data that are frequency tables of counts. For instance, assuming that the number of followers that each Instagram account has is known, if we group them by the number of followers, and then we count how many accounts each group has, it gives place to the frequencies of frequency table. 

By taking logarithm in both sides of (\ref{eqn_zipf_pmf}) one has that when the probabilities are  plotted in log-log scale they show a straight line with a slope equal to $-\alpha$ and an intercept equal to $\log(\zeta(\alpha))$. Figure \ref{fig:dzipf} shows the probabilities of the Zipf for different values of the $\alpha$ parameter. On the left-hand side the probabilities are shown in standard scale, and on the right-hand side in log-log scale. Observe that when the $\alpha$ parameter increases, the probabilities concentrates at the low values.

\begin{figure}[!ht] 
\center
\includegraphics[width=.40\linewidth, angle=-90]{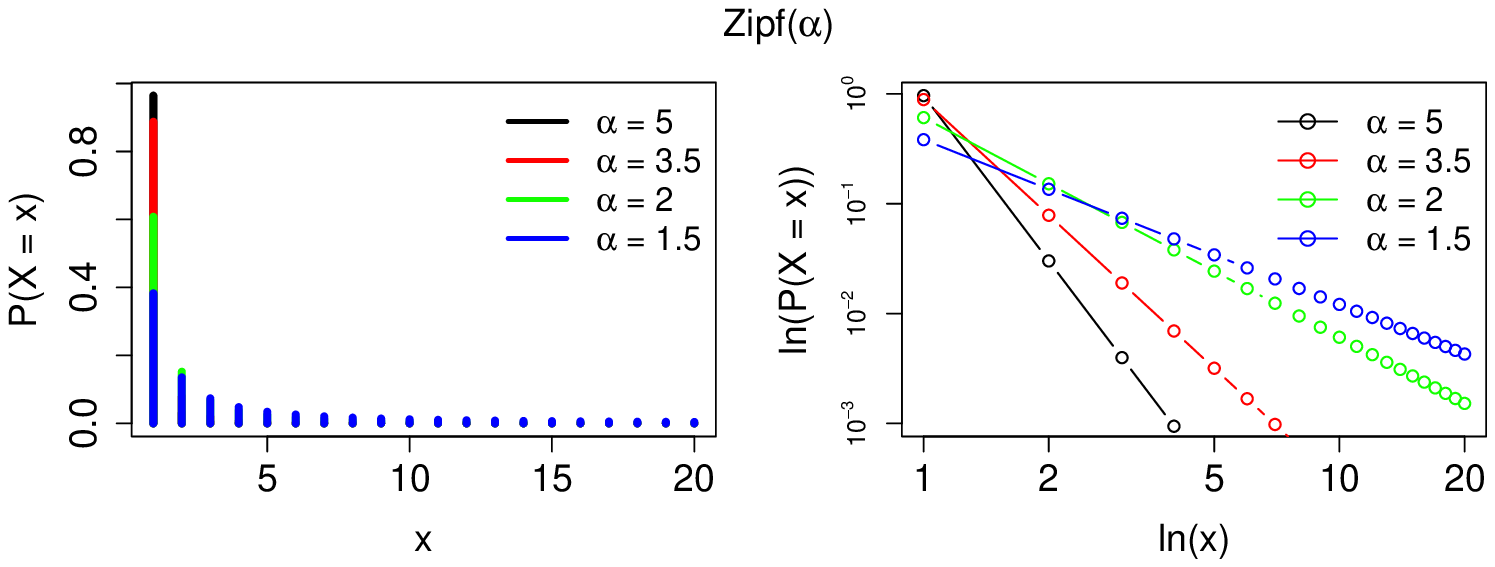} 
\caption{PMFs of the Zipf distribution for $\alpha = 1.5, 2, 3.5$ and $5$. On the left-hand side: normal scale. On the right-hand side: log-log scale.}  
\label{fig:dzipf} 
\end{figure}

The \textit{k-th} moment of the Zipf, $k \in \mathbb{Z}^+$ is equal to:
\begin{equation}
\label{eq:zipfmoment}
E[X^k] = \sum_{x=1} ^{+\infty} \frac{x^k x^{-\alpha}}{\zeta(\alpha)} = \frac{\zeta(\alpha - k)}{\zeta(\alpha)},
\end{equation}
and thus, it is finite if, and only if, $\alpha > k + 1$ because $\zeta(\alpha -k)$ needs to be finite. In particular, the first moment only exists if $\alpha > 2$ and in that case, it is equal to:
\begin{equation}
\label{eq:zipfE}
E[X] =  \frac{\zeta(\alpha - 1)}{\zeta(\alpha)}, \, \alpha > 2.
\end{equation}
\noindent Since the variance depends on the second moment of the distribution, it is finite if, and only if, $\alpha > 3$ and when it exists, it is equal to:
\begin{equation}
\label{eq:zipfVar}
Var[X] = E[X^2] - (E[X])^2 = \frac{\zeta(\alpha - 2) \, \zeta(\alpha) - \zeta(\alpha - 1)^2}{\zeta(\alpha)^2}\,, \alpha > 3.
\end{equation}

The maximum likelihood estimator of the $\alpha$ parameter of the Zipf distribution  may be obtained numerically by solving the following equation:

\begin{align*}
E[\log(X)] = \frac{1}{n}\sum_{i = 1} ^n \log(x_i) = \overline{\log(x)},
\end{align*}
which is equivalent to applying the moment method estimation to the logarithm of the variable.

In \cite{visser2013zipf}, it is proved that the Zipf distribution is the discrete uni-parametric distribution with support on the strictly positive integer values that has maximum Shannon entropy, for a fixed value of $\overline{\log(x)}$. The probability distribution with maximum entropy is assumed to be the best distribution since it is the one maximizing the amount of information or the uncertainty about what is unknown \citep[][p.~1779]{floudas2009pardalos}. 

The probability generating function (PGF) of a Zipf distributed r.v., is equal to:
\begin{equation} \label{eq:pgfzipf}
G_{X}(z)=E(z^X)=\sum_{x=1}^{+\infty} \frac{z^x x^{-\alpha}}{\zeta(\alpha)}=\frac{Li_{\alpha} (z) }{Li_{\alpha}(1)},\, |z| < 1 \text{ and } \alpha >1,  
\end{equation}
where $Li_{\alpha} (z)$ is the \textit{Polylogarithm function} or \textit{Li function of order $\alpha$}, and it is equal to:
\begin{align}
\label{eqn_polylog}
Li_{\alpha} (z)=\sum_{x=1}^{+\infty} \frac{z^x}{x^{\alpha}}.
\end{align}

\noindent Important to observe that $Li_{\alpha}(1) = \zeta(\alpha)$ and thus, the \textit{Li} function may be seen as an extension of the Riemann zeta function. In \cite{valero2022zipf} it is proved that using  analytic prolongation, the $Li$ function is defined throughout the complex plane. The Polylogarithm function may be expressed in terms of the integral Bose-Einstein distribution as follows:

\begin{equation} \label{eq:polylogexpression}
 Li_{\alpha} (z)=\frac{1}{\Gamma (\alpha)} \int_0^{+\infty}   \frac{t^{\alpha-1}}{\frac{exp(t)}{z}-1}\, dt,
\end{equation}
for $Re(\alpha)>0$, and all $z$ except for $z$ that are real and larger or equal to one.

There are numerous examples where researchers have argued for the suitability of the Zipf distribution in modeling various types of data. For instance, \cite{ectors2018exploratory} demonstrated that the Zipf distribution appears in the frequency of daily activities, which can be directly applied to validate travel demand models. Similarly, \cite{wang2017zipf} showed that this distribution can also be used to adjust the vulnerable portion of user-chosen passwords' in cybersecurity. In \cite{chacoma2021word}, the Zipf law is applied to analyze word frequency-rank relationships in English texts based on word class. Additionally, \cite{chen2021exploring} used the Zipf’s exponent to examine the level of urbanization in a large country, while \cite{asif2021statistical} demonstrated the suitability of the Zipf distribution for fitting the upper tail of the Forbes list of world billionaires. In \cite{perc2010zipf} it is shown that, the citation distribution derived from individual publications in Slovenia's current research information system follows a Zipf distribution with an $\alpha$ parameter that usually lies between $2$ and $3$. More recently, research has explored the concept of the Zipf random set (see \cite{lifshits2022two, lifshits2023probabilistic, lifshts2024intersections}).  Additionally, \cite{wang2023zipf} uses Zipf's law to classify words as ``common'' or ``rare'', enabling a similar classification for sentences, which is incorporated into text generations models. This papers also discuses the potential of Zipf's law in the development of AI systems. Generally, the Zipf distribution is considered a plausible probability model in scenarios where small observations occur frequently and large observations are less common. 

For many years, Zipf's law has been considered appropriate for modeling the degree distribution of networks, giving rise to what is known as a ``scale-free'' network. However, this assumption has been widely debated and questioned in recent years, and it has been shown that in many cases, the Zipf's law is only suitable for modeling the data tail. In \cite{duarte2020zipf} and \cite{valero2022zipf}, two different generalizations of Zipf's law are introduced that are able to model degree sequences across its entire range. In \cite{lee2024degree}, a discrete mixture distribution of the Zipf-Polylog and the Generalized Pareto distribution is proposed to address this issue.

\section{The Zipf as a mixture distribution} \label{sec_zipfContrib}
Given a parametric probability distribution, one way to generalize it is by assuming that one of its parameters is a r.v. instead of being constant, which gives place to what is called a \textit{mixture distribution}. The distribution of the parameters is known as ``mixing distribution''. A mixing distribution is required, for instance, to adapt the heterogeneity that usually exists among experimental units. One property of any mixed distribution is that its variance is always larger than the variance of the initial distribution with the same mean. 
See Chapter 8 of \cite{johnson2005univariate} for more information about mixing distributions of discrete r.v.'s. 

Theorem 1 of \cite{hill1975stronger} proves that under certain regularity assumptions, the Zipf distribution is asymptotically the limit of a mixing geometric distribution for different mixing distributions. However, their mixing expression is true only in the tail of the distribution, that is for large values of $x$.

This section proves that the Zipf distribution is: i) a mixture of geometric distributions and ii) a MZTP distribution. In both cases, the mixing distribution is specified analytically. As a consequence of being a MZTP distribution, its variance is larger than the one of a zero-truncated Poisson distribution with the same mean. Important to note that our results are not asymptotic, and thus do not require a large value of $x$. The proofs of Theorems 1 and 2 are based on the PGF of the associated distributions, as working with the PDF has been instrumental in deriving the analytical expressions for the mixing distributions. However in Appendix \ref{app_proof}, Theorem \ref{prop:teo-geozipf} and point a) of Theorem \ref{prop:teorema2} are proven using the PMF, establishing an alternative validation.

\noindent Let us  note that if $N^{zt}$ denotes the zero-truncated version of a r.v. $N$ with a Poisson$(\lambda)$ distribution, given that the PGF of the Poisson$(\lambda)$ is equal to $h_N(z)=e^{\lambda (z-1)}$, one has that,  
\begin{equation} \label{eq:pgfPositivePoisson}
h_{N^{zt}} (z;\lambda)= \frac{e^{\lambda z}-1}{e^{\lambda}-1}.
\end{equation} 

As a consequence of the fact that $\lim_{\lambda \to 0} h_{N^{zt}} (z;\lambda)=z$, it is possible to consider $[0,+\infty)$ as the parameter space of the $N^{zt}$ distribution, where $\lambda=0$ corresponds to the degenerate distribution at one. 

\noindent In Theorem 3 of \cite{valero2022zipf} it is proved that the geometric distribution with parameter $p\in(0,1)$ and domain $\{1,2,\cdots\}$ is an MZTP distribution with mixing distribution:
\begin{equation}\label{eq:mixinggeo}
	f(\lambda;p)=\frac{p}{(1-p)^2} e^{-\lambda/(1-p)} (e^{\lambda}-1),\,\,\, \lambda\in (0,+\infty).
\end{equation}
Thus, taking into account that the PGF of a mixture distribution is the integral of the PGF of mixed distribution multiplied by the probability density function of the mixing distribution, by (\ref{eq:pgfPositivePoisson}) one has that:

\begin{equation} \label{eq:pgfgeom-mixtura}
\frac{ p z}{1-(1-p) z}=\int_0^{+\infty} \frac{e^{\lambda z}-1}{e^{\lambda}-1} f(\lambda; p)\, d \lambda.    
\end{equation}

Next two theorems prove that the Zipf distribution is a mixture of geometric distributions (Theorem \ref{prop:teo-geozipf}) as well as a MZTP distribution (Theorem \ref{prop:teorema2}).

\begin{theorem}
\label{prop:teo-geozipf}
The Zipf($\alpha$) distribution is a mixture of  geometric distributions with domain $\{1,2,3,\cdots\}$ and parameter $s = -log(1-p)$, with mixing distribution:
\begin{equation} \label{eq:mixinggeozipf}
f(s;\alpha)=\frac{s^{\alpha-1}}{(e^{s}-1) \zeta(\alpha) \Gamma(\alpha)},\,\,\,  s>0, \textrm{and}\,\, \alpha>1.
\end{equation}    
\end{theorem}

\begin{proof}
The PGF of the geometric distribution defined in strictly positive integers, as a function of the parameter $s$ is equal to:
\begin{align*}
\frac{pz}{1-qz} = \frac{(1-e^{-s})z}{1-e^{-s}z} = \frac{(e^s -1)z}{e^s -z}.
\end{align*}

Observe that, with the new parametrization, the parameter space is $(0, +\infty)$ instead of $(0, 1)$. Thus, taking into account the integral expression of the $Li$ function (\ref{eq:polylogexpression}), its mixture distribution with mixing distribution as defined in (\ref{eq:mixinggeozipf}) is equal to:
\begin{align}
\label{eq:formula1}
&\int_0^{+\infty} \frac{(e^{s}-1) z}{e^{s}-z} \frac{s^{\alpha-1}}{(e^{s}-1) \zeta(\alpha) \Gamma(\alpha)} ds =  \frac{z}{\zeta(\alpha) \Gamma(\alpha)} \int_0^{+\infty} \frac{s^{\alpha-1}}{e^{s}-z}ds \nonumber \\[5 pt]
&= \frac{1}{\zeta(\alpha)} \int_0^{+\infty} \frac{s^{\alpha - 1}}{\frac{e^{s}}{z} -1}ds
= \frac{Li_{\alpha} (z)}{Li_{\alpha}(1)}, 
\end{align}
\noindent which, by  (\ref{eq:pgfzipf}), is the PGF of the Zipf($\alpha$) distribution.
%$\hfill{\square}$
\end{proof}

Figure \ref{fig:mixingDist} contains the plot of the mixing distribution of Theorem \ref{prop:teo-geozipf}, as a function of $s$ (on the left-hand side) and as a function of $p$ (on the right-hand side), for different $\alpha$ values.

\begin{figure}[!ht] 
\center
\includegraphics[width=.40\linewidth, angle=-90]{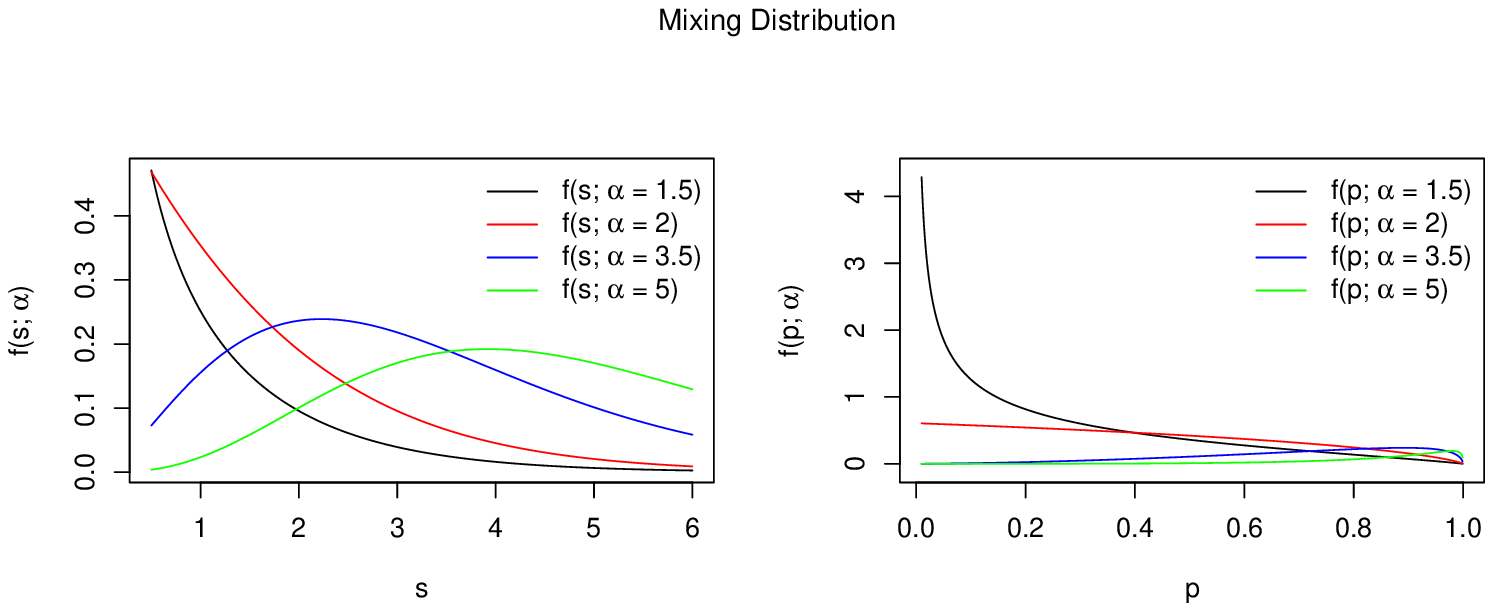} 
\caption{The mixing distribution of Theorem \ref{prop:teo-geozipf}, as a function of $s$ (on the left-hand side) and as a function of $p$ (on the right-hand side), for different $\alpha$ values.}  
\label{fig:mixingDist} 
\end{figure}

Theorem 1 of \cite{valero2010zero} characterizes the families of distributions with finite mean that are ZTMP,  based on their PGF. The theorem states that a PGF $h(z)$ is the PGF of a ZTMP distribution if and only if it verifies that:
\begin{itemize}
	\item[(a)] $h(0)$=0, $h(1)=1$ and $h'(1)<+\infty$;
	\item[(b)] it is analytical in $(-\infty,1)$;
	\item[(c)] all the coefficients of the series expansion of $h(z)$ around any point $z_0 \in (-\infty,1)$ are strictly positive, except for the constant term that may be negative or zero; and
	\item[(d)] $\lim_{z \to -\infty} h(z)=-L$, with $L$ being a finite strictly positive number.
\end{itemize}
Theorem 2 of the same paper establishes that the PGFs of MZTP distributions need to verify the first three conditions of Theorem 1, but not the last one. As a consequence, any ZTMP distribution is an MZTP distribution, but not the other way around. The characterizations are also true if the distribution has no finite mean. The next theorem establishes that the Zipf belongs to the MZTP class, but not to the ZTMP class. 

\begin{theorem} \label{prop:teorema2}
The Zipf($\alpha$) distribution verifies that: 
	\begin{itemize}
		\item[a)] it is an MZTP distribution with mixing distribution equal to:
		\begin{equation} \label{eq:mixingzipf}
		f(\lambda; \alpha)=   \frac{(e^{\lambda}-1) \int_0^{+\infty}   e^{s-\lambda e^{s}}  s^{\alpha-1} ds }{\Gamma(\alpha) \zeta(\alpha) },\,\, \lambda>0,
		\end{equation}
		\item[b)] it is not a ZTMP distribution.
	\end{itemize} 
\end{theorem}
\begin{proof}

Taking into account (\ref{eq:pgfzipf}), proving a) is equivalent to seeing that
$$
\frac{Li_{\alpha} (z)}{Li_{\alpha} (1)} = \int_0^{+\infty} \frac{e^{\lambda z}-1}{e^{\lambda}-1} \, f(\lambda; \alpha)\, d \lambda,
$$ 
with $f(\lambda; \alpha)$ defined as in (\ref{eq:mixingzipf}). First observe that, as a consequence of Theorem \ref{prop:teo-geozipf}, from (\ref{eq:formula1}) we have:
\begin{equation} \label{eq:lio1}
\frac{Li_{\alpha} (z)}{Li_{\alpha} (1)} = \int_0^{+\infty}     \frac{(e^{s}-1) z}{e^{s}-z}    \frac{s^{\alpha-1}}{(e^{s}-1) \zeta(\alpha)\Gamma(\alpha)} ds.
\end{equation}
Now, rewriting (\ref{eq:pgfgeom-mixtura}) with the $s$ parametrization we also have:
\begin{equation} \label{eq:lio2}
\frac{ (e^{s}-1) z}{e^{s}-z}= \int_0^{+\infty} \frac{ e^{\lambda z}-1}{e^{\lambda}-1} f^*(\lambda;s) d \lambda,
\end{equation}
where $$f^*(\lambda;s)=f(\lambda;1-e^{-s})=e^{s} (e^{s}-1) e^{-\lambda e^{s}} (e^{\lambda}-1).$$
Substituting (\ref{eq:lio2}) in (\ref{eq:lio1}) gives that:
\begin{eqnarray} 
	\frac{Li_{\alpha} (z)}{Li_{\alpha} (1)} & = & \int_0^{+\infty}  \Big[  \int_0^{+\infty} \frac{ e^{\lambda z}-1}{e^{\lambda}-1} f^*(\lambda;s)\, d \lambda \Big]  \frac{s^{\alpha-1}}{(e^{s}-1) \zeta(\alpha) \Gamma(\alpha)} d s \nonumber \\[5 pt] \nonumber & = &  
	\int_0^{+\infty}   \frac{ e^{\lambda z}-1}{e^{\lambda}-1} \int_0^{+\infty}  \Big[  f^*(\lambda; s) \frac{s^{\alpha-1}}{(e^{s}-1)} \frac{1}{\zeta(\alpha) \Gamma(\alpha)}     \Big]  ds\, d\lambda \noindent \\[5 pt] \nonumber & = &   	\int_0^{+\infty}   \frac{e^{\lambda z}-1}{e^{\lambda}-1}  \Big[ \frac{e^{\lambda}-1}{\zeta(\alpha) \Gamma(\alpha)}    \int_0^{+\infty} e^{s-\lambda e^{s}} s^{\alpha-1}\, ds \Big]\,d\lambda,
\end{eqnarray}
which proves $a)$. To prove $b)$, it is necessary to see that condition d) of Theorem 1 of \cite{valero2010zero} is not satisfied. But, this is the case since:
\begin{eqnarray}
\lim_{z \to -\infty} \frac{Li_{\alpha} (z)}{Li_{\alpha} (1)} & = &\lim_{z \to -\infty} \frac{z}{\Gamma(\alpha) Li_{\alpha} (1)} \int_0^{+\infty} \frac{t^{\alpha-1}}{e^{t}-z} \, dt \nonumber \\   & = &  \lim_{z \to -\infty} \frac{1}{\Gamma(\alpha) \zeta(\alpha)} \int_0^{+\infty} \frac{ t^{\alpha-1}  }{ \frac{e^{t}}{z}-1}  =-\infty.
\end{eqnarray}
%$\hfill{\square}$
\end{proof}
Figure \ref{fig:mixingDist} contains the plots of the mixing distribution of Theorem \ref{prop:teorema2} as a function of $\lambda$ for $\alpha$ equal to $1.1, 1.5, 2, 3.5$ and $5$. As it can be seen, as $\alpha$ increases the density function tends to have less slope. When the value of $\alpha$ is small, the probability of small $\lambda$'s increase significantly.

The interpretation of Theorem \ref{prop:teorema2} is as follows: if we take an MP distribution and we truncate it at zero, in order that it takes values from one to infinity,  we will never obtain the Zipf distribution. However, if we first truncate at zero a Poisson distribution, and then the parameter of the zero-truncated Poisson is considered to follow the probability distribution defined at (\ref{eq:mixingzipf}), the result is the Zipf distribution. 

\begin{figure}[!ht] 
\center
\includegraphics[width=.40\linewidth, angle=-90]{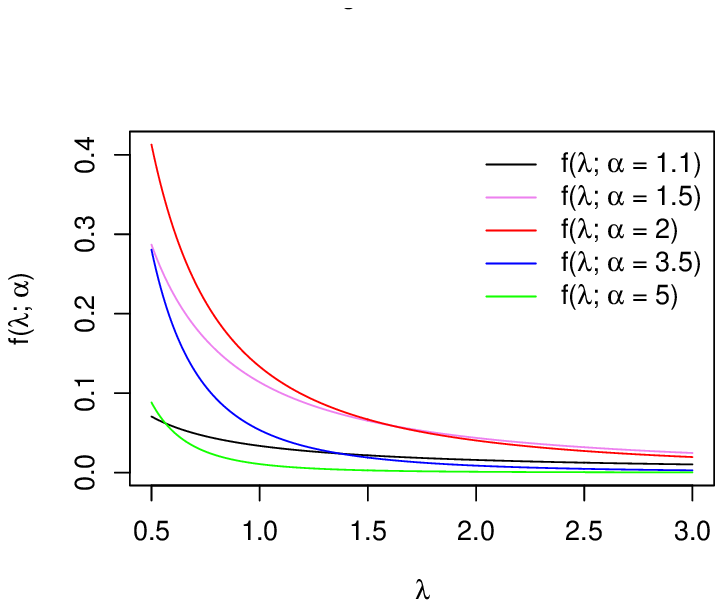} 
\caption{The mixing distribution of Theorem \ref{prop:teorema2} a), as a function of $\lambda$ for $\alpha=$ 1.1, 1.5, 2, 3.5 and 5.}  
\label{fig:mixingDist} 
\end{figure}

In \cite{duarte2020zipf} the Zipf-PSS model is introduced as the Poisson Stopped sum (PSS) version of the Zipf distribution. As a consequence of Theorem \ref{prop:teorema2}  it is possible to prove that the Zipf-PSS is a MP which is done in the following corollary. To that end, one has to consider Theorem 3 of \cite{valero2013poisson}, that proves that  a non-negative integer PSS distribution with finite mean is a MP if, and only if, the zero-truncation of its secondary distribution is an MZTP distribution. As a consequence of this theorem one has the following result:

\begin{corollary}
The Zipf-PSS($\alpha,\lambda$) distribution is a MP for any  $\alpha>2$ and $\lambda>0$.
\end{corollary}

\begin{proof}
Given that $\alpha>2$, the Zipf-PSS($\alpha,\lambda$)  has a finite mean. Moreover, by definition it is a PSS with a Zipf($\alpha$) as a secondary distribution. Given that the Zipf($\alpha$) distribution does not contain the zero value in its support, it is equal to its zero-truncation version, and it is a Mixture of Zero-truncated Poisson distributions as a consequence of Theorem \ref{prop:teorema2}.
%$\hfill{\square}$
\end{proof}

\section{Illustrative Example} \label{section_poc}
This section presents the results of an example that illustrates the veracity of the theoretical result that appears in Theorem \ref{prop:teorema2}. The result presented in Theorem \ref{prop:teo-geozipf} could be checked similarly. The study consist on analyzing each of the 135 chapters of Herman Melville's novel \textit{Moby Dick}. During the recent years, this literary classic has been examined from various perspectives by linguists and statisticians, see, for instance, the works of \cite{espitia2020universal},  \cite{luo2024segmental} and \cite{seibel2024whaleness}.

To begin the analysis,  the novel has been downloaded from Project Gutenberg\footnote{\url{https://www.gutenberg.org}}, a digital library offering free access to thousands of books, most of which are classic literary works. The book was retrieved in HTML format from \url{https://www.gutenberg.org/files/2701/2701-h/2701-h.htm}. Then, the content has been parsed  to separate the 135 chapters, which are analyzed in this section.
To normalize the extracted content, each chapter underwent a preprocessing step that included: i) expanding contractions (e.g., ``don’t'' to ``do not''), ii) removing stop words, punctuation marks, extra white spaces, and non-alphanumeric characters. However, hyphenated words like ``sea-sick'' were retained, as they are standard forms of bi-grams. Finally, each word was reduced to its lemma (the canonical or base form). Interjections or expressive sounds like ``woo-hoo'' or ``wa-hee'' were kept as they appear in the text.  
The final step of the initial processing involved obtaining word counts and, from those, generating the corresponding frequency-of-frequencies table for each chapter. The table shows the frequencies of word that appear in the chapter, along with the number of words that occur at each frequency. In \ref{app_poc} are shown the frequency of frequencies tables of Chapter 1 (\ref{tab:freq_chapter1}) and Chapter 135 (\ref{tab:freq_chapter135}) respectively. These tables are the base of the posterior statistical analysis.

After the preprocessing step, the analysis has been performed for each chapter separately, and it is configured as follows: first, we used Maximum Likelihood Estimation method  to determine the $\hat{\alpha}$ parameter estimation for the Zipf distribution. Next, for each frequency value, $i$, in the frequency of frequencies table, we estimate the $\hat{\lambda_i}$ parameter by equating the mean value of Zero-truncated Poisson distribution to the observed frequency in the chapter (method of moments). This is equivalent to solve the following equation:
\begin{equation*}
    n_i = \frac{\lambda_i}{1 - exp(-\lambda_i)},\,\,\, i=1\cdots m
\end{equation*}
where $m$ is the maximum frequency of a word in the chapter, and
$n_i$ stands for the number of words that have a frequency equal to $i$. Observe that all the words with the same frequency in the table will have the same $\lambda$ parameter estimation.

\noindent After this process,  a sequence of $\lambda$ values is obtained which, according with Theorem \ref{prop:teorema2}, should follow the distribution that appears in (\ref{eq:mixingzipf}). To verify this, the cumulative empirical distribution has been calculated for each chapter and compared with the theoretical cumulative distribution derived from the density function specified in (\ref{eq:mixingzipf}) with $\alpha$ equal to $\hat{\alpha}$. At that point, the Kolmogorov-Smirnov (KS) test has been performed using the \textbf{ks.test} function in the statistical software R. 

Table \ref{tab_results} in \ref{app_poc} provides, for all analyzed chapters, the chapter number along with the number of words, minimum and maximum frequencies, the count of non-zero frequencies, the $\hat{\alpha}$ estimate and its confidence interval, the KS statistic value, and the associated p-value.

Observe that $10$ out of the $135$ chapters have more that $1000$ words, being Chapter $54$ the most extensive with $1883$ as a total number of words. The smaller chapter is Chapter $120$ with only $71$ words and just $5$ different frequencies (from $1$ to $5$). The frequencies of words observed largely depend on the chapter under study. For instance, Chapters $16$ and $32$ have similar number of words, but the higher frequency of a word is equal to $49$ in Chapter $16$ and $142$ (nearly three times larger) in Chapter $32$. As a consequence of that, the frequency of frequencies tables are very different in dimension between the chapters. There are $6$ chapters with just $4$ different frequencies different from zero (Chapters 30, 38, 95, 97, 114 and 122) and a total of $14$ chapters, out of $135$, with a total number of frequencies different from zero smaller or equal to $5$. The larger table is the one associated to Chapter $54$ which has a total of $27$ different rows. Across the 135 chapters analyzed, the null hypothesis of the  KS tests is never rejected at a significance level of $0.05$, and it is rejected only once (Chapter 135) when the significance level is increased to $0.10$. In consequence, the data do not leads us to reject that the sequence of $\lambda$ values, for a given chapter, follows the distribution specified in Theorem \ref{prop:teorema2}.

The left-hand side of Figure \ref{fig:poc} displays, for Chapters $1$ and $135$, the values of the table of frequencies of frequencies in log-log scale, along with the fit obtained with a Zipf($\hat{\alpha}$) distribution. On the right-hand it appears the empirical cumulative probabilities alongside the theoretical cumulative probability distribution function of the distribution specified in (\ref{eq:mixingzipf}) with $\alpha=\hat \alpha$, for the same two chapters. Note that function (\ref{eq:mixingzipf})  is defined only for $\lambda >0$ and so it is its cumulative distribution function, however  $\lambda=0$ is an observed value since it is the $\lambda$ value that is associated to the words that appear once in the chapter. In both cases it is possible to see that both cumulative distribution functions are very similar. Similar behaviour is observed for the rest of the Chapters.

\begin{figure}[h!]
  \centering  
  % First image
  \begin{minipage}{\linewidth}
    \centering
    \includegraphics[width=.40\linewidth, angle=-90]{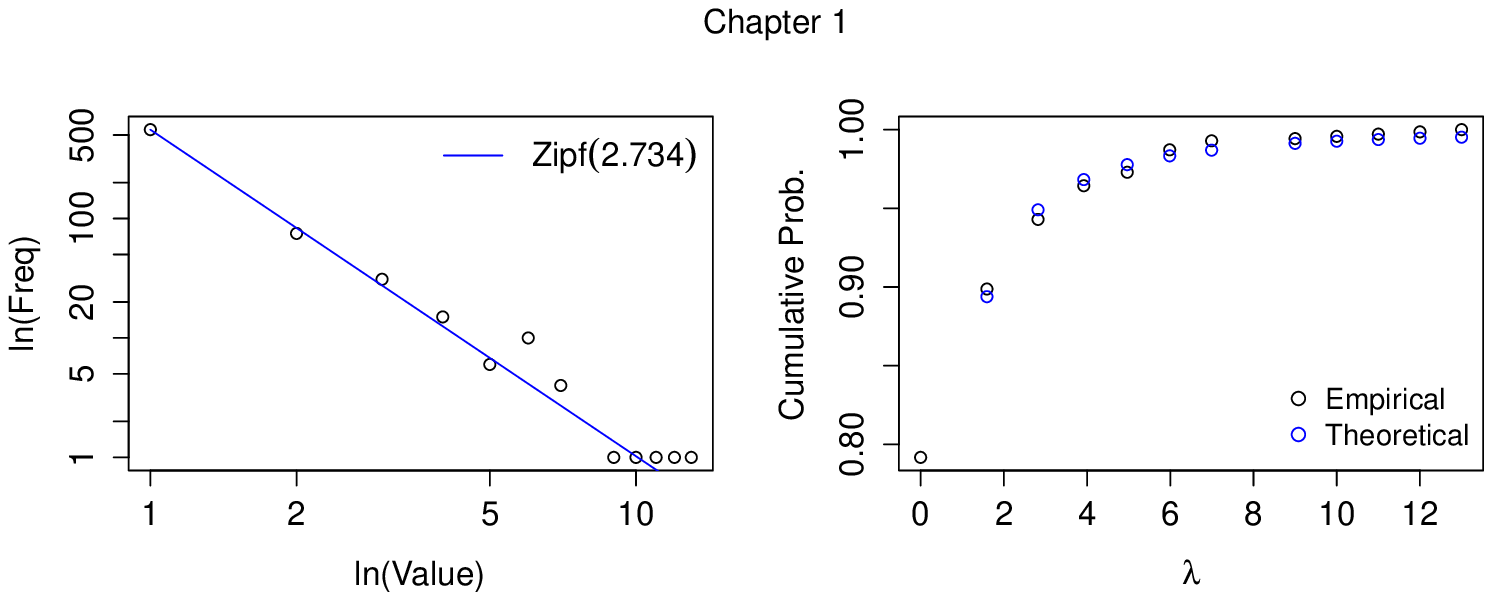}
    %\caption{First image.}
    %\label{fig:first}
  \end{minipage}
\vspace{0.5em} % Adds vertical space between images
  % Second image
  \begin{minipage}{\linewidth}
    \centering
    \includegraphics[width=.40\linewidth, angle=-90]{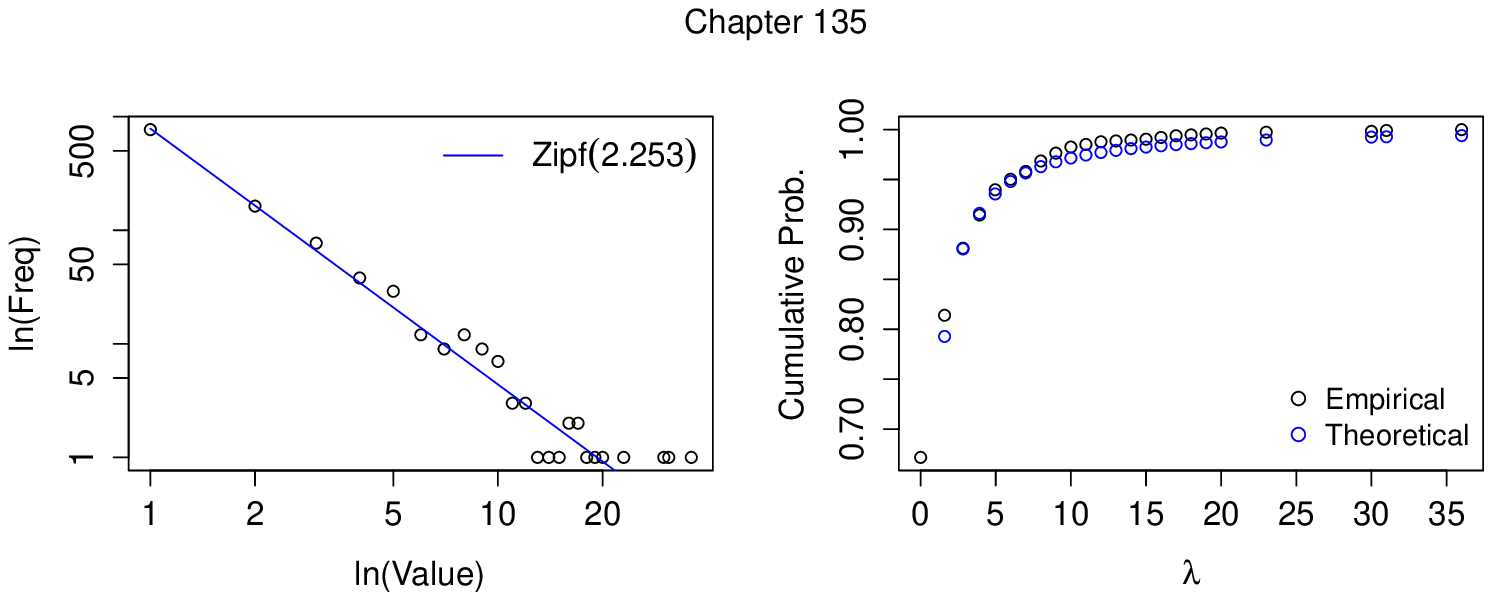}
   % \caption{Second image.}
   % \label{fig:second}
  \end{minipage}
  \caption{On the left-hand side, the frequencies of frequencies are shown in log-log scale, along with the fit obtained using a Zipf($\hat{\alpha}$) distribution. On the right-hand side, the empirical cumulative probabilities are displayed, along with the theoretical cumulative probability distribution function for chapters 1 (top) and 135 (bottom).}
  \label{fig:poc}
\end{figure}

\section{Conclusions} \label{section_conclusions}
Understating the underling mechanisms that generate real data is always a valuable tool for understanding the nature of the data under analysis. 
In this work, we have analyzed the Zipf distribution through a mixture analysis, demonstrating that it can be obtained by generating data from a geometric distribution with a probability parameter that, instead of being constant, follows the probability distribution specified in Theorem \ref{prop:teo-geozipf}. Additionally, it can be obtained by generating data from a zero-truncated Poisson distribution, where the mean parameter varies from observation to observation according to the probability law specified in Theorem \ref{prop:teorema2}. It's essential to note that although it is an MZTP, it is not a ZTMP. Consequently, truncating a MP distribution at zero will never yield the Zipf distribution as a result. 

The Poisson scenario is checked through an extensive analysis of the $135$ chapters of the Moby Dick novel where, according to the KS test with an $\alpha=0.05$, all the chapters align with the result obtained in Theorem \ref{prop:teorema2}.

\section*{Declarations}
The authors have no conflicts of interest to declare that are relevant to the content of this article.

\section*{Funding}
Research partially supported by the projects AGAUR 2021 SGR 00613, UPC AGRUPS-2024 and PID2023-148158OB-I00 from the Spanish Ministry of Science Innovation and Universities.

\begin{appendices}
\section{Alternative Proofs to Theorem \ref{prop:teo-geozipf} and point a) of Theorem \ref{prop:teorema2}}
\label{app_proof}

In the following, the proofs of Theorem \ref{prop:teo-geozipf} and point a) of Theorem \ref{prop:teorema2} based on the PMF instead of the PGF are presented. As observed, the proof  of Theorem \ref{prop:teo-geozipf} is slightly shorter when using PMFs, while the proof of point a) of Theorem 2 is comparable in length to the one provided in Section \ref{sec_zipfContrib}.

\subsection{Proof of Theorem 1}
The PMF of the geometric distribution defined in $\{1,2,3,\cdots \}$ parameterized in terms of $s=-\log(1-p)$ is equal to $e^{-s (x-1)} \cdot (1-e^{-s})$. Thus, is is necessary to prove that:
\begin{equation} \label{eq:altproofT1}
\frac{x^{-\alpha}}{\zeta(\alpha)}=\int_0^{+\infty} e^{-s (x-1)} \cdot (1-e^{-s}) \cdot f(s;\alpha) ds
\end{equation}
where $f(s;\alpha)$ is defined in (\ref{eq:mixinggeozipf}). Substituting $f(s;\alpha)$ by his expression, the right-hand side of  (\ref{eq:altproofT1}) turns out to be equal to:
\begin{align*}
\int_0^{+\infty} e^{-s (x-1)} \cdot (1-e^{-s}) \cdot f(s;\alpha) ds&=
\int_0^{+\infty} e^{-s (x-1)} \cdot (1-e^{-s}) \cdot \frac{s^{\alpha-1}}{(e^s-1) \zeta(\alpha) \Gamma(\alpha)} ds
\\&=
\frac{1}{\zeta(\alpha)} \int_0^{+\infty} e^{-s x} \frac{s^{\alpha-1}}{\Gamma(\alpha)} ds,
\end{align*}
performing the change of variables $t= s x$, one has that:
$$
\int_0^{+\infty} e^{-s (x-1)} \cdot (1-e^{-s}) \cdot f(s;\alpha) ds=\frac{x^{-\alpha}}{\zeta(\alpha)} \int_0^{+\infty} \frac{ e^{-t} t^{\alpha-1}}{\Gamma(\alpha)} ds =\frac{x^{-\alpha}}{\zeta(\alpha)}
$$
$\hfill{\square}$

\subsection{Proof of point a) of  Theorem 2}
Taking into account that the PMF of the zero-truncated Poisson distribution with parameter $\lambda$ is equal to $e^{-\lambda} \lambda^x/((1-e^{-\lambda})\, x!)$,  proving the theorem is equivalent to seeing that
\begin{equation} \label{eq:altproofT2}
    \frac{x^{-\alpha}}{\zeta(\alpha)}=\int_0^{+\infty} \frac{e^{-\lambda} \lambda^x}{(1-e^{-\lambda})\, x!} \cdot f(\lambda; \alpha) d \lambda
\end{equation}

where $f(\lambda;\alpha)$ is defined in (\ref{eq:mixingzipf}). From Theorem (\ref{prop:teo-geozipf}) one has that:
\begin{equation} \label{eq:eqT2-1}
\frac{x^{-\alpha}}{\zeta(\alpha)}=\int_0^{+\infty} e^{-s (x-1)} \cdot (1-e^{-s}) \cdot \frac{s^{\alpha-1}}{(e^s-1) \zeta(\alpha) \Gamma(\alpha)} ds
\end{equation}
From Theorem 3 of \cite{valero2022zipf},  parametrizing  the geometric distribution with $s$, one has that:
\begin{equation} \label{eq:eqT2-2}
e^{-s (x-1)} \cdot (1-e^{-s})= \int_0^{+\infty} \frac{e^{-\lambda} \lambda^x}{(1-e^{-\lambda})\, x!} \cdot \frac{1-e^{-s}}{e^{-2 s}} e^{-\lambda e^s} (e^{\lambda}-1) d \lambda
\end{equation}
Substituting (\ref{eq:eqT2-2}) in (\ref{eq:eqT2-1}), and  taking into account that $e^{-\lambda} (e^{\lambda}-1)/(1-e^{-\lambda})=1$, one has that:
\begin{align*}
    \frac{x^{-\alpha}}{\zeta(\alpha)} & =  \int_0^{+\infty} \Big(   \int_0^{+\infty} \frac{ \lambda^x}{ x!} \cdot \frac{1-e^{-s}}{e^{-2 s}} e^{-\lambda e^s}  d \lambda    \Big) \cdot \frac{s^{\alpha-1}}{(e^s-1) \zeta(\alpha) \Gamma(\alpha)} ds \\ &=  \int_0^{+\infty} \frac{\lambda^x}{x!} \Big(\int_0^{+\infty} \frac{e^{-\lambda e^s+s} s^{\alpha-1}}{ \zeta(\alpha) \Gamma(\alpha)  } ds \Big) d \lambda \nonumber \\
&=   \int_0^{+\infty} \frac{ e^{-\lambda} \lambda^x}{(1-e^{-\lambda}) x!} \Big(\int_0^{+\infty} \frac{ (e^{\lambda}-1) e^{-\lambda e^s+s} s^{\alpha-1}}{ \zeta(\alpha) \Gamma(\alpha)  } ds \Big) d \lambda 
\end{align*}  

$\hfill{\square}$

\section{Frequency of frequencies tables of Chapters $1$ and $135$, the chapters used to illustrate the analysis in Section \ref{section_poc}}
\label{app_poc}
%Frequency of frequencies tables of Chapters $1$ and $135$, the chapters used to illustrate the analysis in Section \ref{section_poc}.
%\vspace{-5\textwidth}
\begin{table}[htp!]
\label{tab:freq}
\centering  
\begin{minipage}{0.45\linewidth}  % First table taking up 45% of the page width
\centering
%\vspace{-.9\textwidth}
\begin{tabular}{cc}
  \hline
  \textbf{Value} & \textbf{Freq} \\
  \hline
  1  & 555 \\
  2  & 75  \\
  3  & 31  \\
  4  & 15  \\
  5  & 6   \\
  6  & 10  \\
  7  & 4   \\
  9  & 1   \\
  10 & 1   \\
  11 & 1   \\
  12 & 1   \\
  13 & 1   \\
  \hline
  & 701\\
  \hline
\end{tabular}
\caption{Frequency of Frequencies table associated with Chapter 1.}
\label{tab:freq_chapter1}
\end{minipage}
\hspace{0.05\linewidth}  % Horizontal space between the two tables
\begin{minipage}{0.45\linewidth}  % Second table taking up 45% of the page width
\centering
\begin{tabular}{cc}
  \hline
  \textbf{Value} & \textbf{Freq} \\
  \hline
  1  & 769 \\
  2  & 163 \\
  3  & 77  \\
  4  & 38  \\
  5  & 29  \\
  6  & 12  \\
  7  & 9   \\
  8  & 12  \\
  9  & 9   \\
  10 & 7   \\
  11 & 3   \\
  12 & 3   \\
  13 & 1   \\
  14 & 1   \\
  15 & 1   \\
  16 & 2   \\
  17 & 2   \\
  18 & 1   \\
  19 & 1   \\
  20 & 1   \\
  23 & 1   \\
  30 & 1   \\
  31 & 1   \\
  36 & 1   \\
  \hline
& 1145 \\
\hline
\end{tabular}
\caption{Frequency of Frequencies table associated with Chapter 135.}
\label{tab:freq_chapter135}
\end{minipage}
\end{table}

Results of the analysis conducted in Section \ref{section_poc}. Table \ref{tab_results} summarizes the results obtained after analysing the 135 chapters included in the Novel Moby Dick by Herman Melville.

\begin{longtable}{ccccccccc}
\caption{Results obtained in the analysis. They include statistics such as the minimum and maximum frequencies, the number of different words, the number of words after preprocessing, the $\hat{\alpha}$ parameter of the Zipf distribution along with its confidence interval, the KS statistic, and the associated p-value.} \label{tab_results} \\

 \hline
  \textbf{Chp.} & \textbf{Min} & \textbf{Max} & \textbf{Nbr. of} & \textbf{Nbr. of} & $\boldsymbol{\hat{\alpha}}$ & $\boldsymbol{CI_{\alpha}}$ & \textbf{KS.} & \textbf{KS} \\
  \textbf{Nbr.} & \textbf{Freq} & \textbf{Freq} & \textbf{Diff Freq.} & \textbf{Words} & & & \textbf{Stat.} & \textbf{p-value} \\
  \hline
\endhead

\hline
\multicolumn{9}{r}{Continued on next page} \\
\hline
\endfoot

\hline
\endlastfoot

% Data rows
1 &   1 &  13 &  12 & 701 & 2.73 & (2.5846, 2.8826) & 0.33 & 0.54 \\ 
2 &   1 &   8 &   8 & 494 & 2.79 & (2.6074, 2.9774) & 0.38 & 0.66 \\ 
3 &   1 &  36 &  23 & 1379 & 2.30 & (2.2263, 2.3774) & 0.26 & 0.42 \\ 
4 &   1 &  13 &  10 & 535 & 2.70 & (2.5341, 2.8673) & 0.40 & 0.42 \\ 
5 &   1 &   6 &   6 & 291 & 3.10 & (2.8042, 3.3936) & 0.33 & 0.93 \\ 
6 &   1 &  11 &   8 & 332 & 3.09 & (2.8175, 3.3669) & 0.12 & 1.00 \\ 
7 &   1 &   6 &   6 & 383 & 3.15 & (2.8864, 3.4172) & 0.50 & 0.47 \\ 
8 &   1 &  12 &   8 & 359 & 2.93 & (2.6889, 3.164) & 0.25 & 0.98 \\ 
9 &   1 &  55 &  18 & 996 & 2.48 & (2.378, 2.5848) & 0.22 & 0.78 \\ 
10 &   1 &  10 &   8 & 501 & 2.65 & (2.4832, 2.8147) & 0.38 & 0.66 \\ 
11 &   1 &   5 &   5 & 257 & 3.20 & (2.8649, 3.5314) & 0.40 & 0.87 \\ 
12 &   1 &  10 &   6 & 328 & 2.98 & (2.7202, 3.234) & 0.33 & 0.93 \\ 
13 &   1 &  26 &  10 & 579 & 2.70 & (2.5405, 2.8607) & 0.40 & 0.42 \\ 
14 &   1 &  10 &   6 & 309 & 3.28 & (2.9624, 3.6012) & 0.17 & 1.00 \\ 
15 &   1 &  11 &  10 & 396 & 2.70 & (2.5092, 2.8971) & 0.20 & 0.99 \\ 
16 &   1 &  49 &  28 & 1280 & 2.29 & (2.2166, 2.3724) & 0.18 & 0.77 \\ 
17 &   1 &  26 &  12 & 647 & 2.57 & (2.434, 2.7092) & 0.33 & 0.54 \\ 
18 &   1 &  15 &  10 & 425 & 2.44 & (2.29, 2.5969) & 0.30 & 0.79 \\ 
19 &   1 &  21 &  11 & 299 & 2.31 & (2.1443, 2.4703) & 0.36 & 0.48 \\ 
20 &   1 &   9 &   8 & 313 & 2.77 & (2.538, 2.9943) & 0.25 & 0.98 \\ 
21 &   1 &  17 &  11 & 349 & 2.70 & (2.4958, 2.9087) & 0.18 & 1.00 \\ 
22 &   1 &  19 &  14 & 520 & 2.56 & (2.4053, 2.709) & 0.21 & 0.92 \\ 
23 &   1 &   5 &   5 & 166 & 3.31 & (2.8682, 3.7552) & 0.20 & 1.00 \\ 
24 &   1 &  17 &  11 & 558 & 2.76 & (2.5919, 2.9327) & 0.18 & 1.00 \\ 
25 &   1 &   9 &   5 & 110 & 3.11 & (2.6277, 3.5932) & 0.20 & 1.00 \\ 
26 &   1 &  11 &   9 & 459 & 3.07 & (2.8416, 3.3032) & 0.22 & 0.99 \\ 
27 &   1 &   9 &   9 & 592 & 2.83 & (2.6579, 3.005) & 0.33 & 0.73 \\ 
28 &   1 &  11 &   8 & 502 & 2.79 & (2.608, 2.9747) & 0.38 & 0.66 \\ 
29 &   1 &  11 &   9 & 399 & 2.87 & (2.6527, 3.0867) & 0.22 & 0.99 \\ 
30 &   1 &   7 &   4 & 124 & 3.64 & (3.0209, 4.2522) & 0.25 & 1.00 \\ 
31 &   1 &  12 &  11 & 209 & 2.33 & (2.1278, 2.5241) & 0.18 & 1.00 \\ 
32 &   1 & 142 &  22 & 1262 & 2.31 & (2.2327, 2.3921) & 0.32 & 0.22 \\ 
33 &   1 &   8 &   7 & 368 & 2.98 & (2.7369, 3.2229) & 0.29 & 0.96 \\ 
34 &   1 &  18 &  12 & 726 & 2.81 & (2.6559, 2.9648) & 0.25 & 0.87 \\ 
35 &   1 &  13 &  12 & 837 & 2.68 & (2.5498, 2.8124) & 0.33 & 0.54 \\ 
36 &   1 &  51 &  18 & 803 & 2.41 & (2.3015, 2.5188) & 0.22 & 0.78 \\ 
37 &   1 &  14 &   6 & 211 & 3.31 & (2.9164, 3.702) & 0.17 & 1.00 \\ 
38 &   1 &   4 &   4 & 164 & 3.55 & (3.0361, 4.0551) & 0.25 & 1.00 \\ 
39 &   1 &   6 &   6 & 106 & 3.27 & (2.7317, 3.818) & 0.17 & 1.00 \\ 
40 &   1 &  33 &  12 & 523 & 2.50 & (2.3587, 2.6492) & 0.33 & 0.54 \\ 
41 &   1 &  45 &  16 & 1149 & 2.57 & (2.4696, 2.6763) & 0.19 & 0.95 \\ 
42 &   1 &  38 &  16 & 1155 & 2.63 & (2.5202, 2.7352) & 0.19 & 0.95 \\ 
43 &   1 &   8 &   5 & 107 & 2.92 & (2.4873, 3.3543) & 0.20 & 1.00 \\ 
44 &   1 &  17 &  15 & 654 & 2.68 & (2.5297, 2.826) & 0.20 & 0.94 \\ 
45 &   1 &  46 &  16 & 994 & 2.51 & (2.4081, 2.6205) & 0.25 & 0.72 \\ 
46 &   1 &  13 &   8 & 377 & 2.93 & (2.6959, 3.1601) & 0.25 & 0.98 \\ 
47 &   1 &   6 &   6 & 334 & 2.86 & (2.621, 3.0909) & 0.33 & 0.93 \\ 
48 &   1 &  47 &  19 & 1105 & 2.38 & (2.2947, 2.4761) & 0.21 & 0.81 \\ 
49 &   1 &   8 &   8 & 310 & 3.01 & (2.7389, 3.2783) & 0.25 & 0.98 \\ 
50 &   1 &  12 &   8 & 353 & 2.88 & (2.6489, 3.1139) & 0.25 & 0.98 \\ 
51 &   1 &  11 &  11 & 536 & 2.77 & (2.5973, 2.9475) & 0.27 & 0.83 \\ 
52 &   1 &   6 &   5 & 282 & 3.09 & (2.7901, 3.3845) & 0.40 & 0.87 \\ 
53 &   1 &  11 &  10 & 491 & 2.62 & (2.456, 2.7837) & 0.30 & 0.79 \\ 
54 &   1 &  40 &  27 & 1883 & 2.30 & (2.2392, 2.3687) & 0.26 & 0.33 \\ 
55 &   1 &  43 &  12 & 604 & 2.70 & (2.5409, 2.8537) & 0.25 & 0.87 \\ 
56 &   1 &  25 &   9 & 457 & 2.82 & (2.623, 3.0147) & 0.11 & 1.00 \\ 
57 &   1 &  15 &   6 & 373 & 3.37 & (3.0664, 3.6795) & 0.17 & 1.00 \\ 
58 &   1 &  15 &   8 & 387 & 3.05 & (2.8026, 3.2983) & 0.25 & 0.98 \\ 
59 &   1 &   8 &   6 & 361 & 3.06 & (2.7987, 3.3139) & 0.33 & 0.93 \\ 
60 &   1 &  18 &  11 & 479 & 2.79 & (2.6006, 2.9751) & 0.18 & 1.00 \\ 
61 &   1 &  16 &  12 & 673 & 2.69 & (2.5435, 2.8384) & 0.25 & 0.87 \\ 
62 &   1 &   8 &   6 & 196 & 3.10 & (2.7412, 3.4602) & 0.17 & 1.00 \\ 
63 &   1 &   6 &   6 & 182 & 3.08 & (2.7136, 3.4512) & 0.17 & 1.00 \\ 
64 &   1 &  40 &  19 & 801 & 2.39 & (2.286, 2.5005) & 0.21 & 0.81 \\ 
65 &   1 &  11 &   7 & 335 & 2.78 & (2.5551, 2.9996) & 0.43 & 0.58 \\ 
66 &   1 &   9 &   6 & 255 & 3.28 & (2.93, 3.6332) & 0.17 & 1.00 \\ 
67 &   1 &   7 &   7 & 272 & 2.84 & (2.5848, 3.101) & 0.29 & 0.96 \\ 
68 &   1 &  26 &   9 & 387 & 2.79 & (2.5811, 2.9982) & 0.33 & 0.73 \\ 
69 &   1 &   6 &   5 & 189 & 3.44 & (2.9941, 3.89) & 0.20 & 1.00 \\ 
70 &   1 &  11 &   7 & 355 & 3.05 & (2.7896, 3.3064) & 0.14 & 1.00 \\ 
71 &   1 &  19 &  14 & 701 & 2.58 & (2.4422, 2.7072) & 0.21 & 0.92 \\ 
72 &   1 &  16 &  12 & 520 & 2.75 & (2.5769, 2.9274) & 0.17 & 1.00 \\ 
73 &   1 &  23 &  15 & 559 & 2.42 & (2.2891, 2.5516) & 0.33 & 0.39 \\ 
74 &   1 &  26 &  13 & 482 & 2.66 & (2.4913, 2.8325) & 0.08 & 1.00 \\ 
75 &   1 &  23 &  10 & 398 & 2.71 & (2.5154, 2.9042) & 0.20 & 0.99 \\ 
76 &   1 &  11 &   8 & 325 & 3.31 & (2.9931, 3.6264) & 0.25 & 0.98 \\ 
77 &   1 &  11 &   7 & 265 & 3.09 & (2.781, 3.3942) & 0.14 & 1.00 \\ 
78 &   1 &  18 &  11 & 539 & 2.68 & (2.5205, 2.8486) & 0.27 & 0.83 \\ 
79 &   1 &  12 &   8 & 340 & 2.97 & (2.7201, 3.223) & 0.12 & 1.00 \\ 
80 &   1 &  11 &   8 & 303 & 2.79 & (2.5507, 3.0209) & 0.12 & 1.00 \\ 
81 &   1 &  41 &  18 & 1197 & 2.42 & (2.3269, 2.5058) & 0.39 & 0.13 \\ 
82 &   1 &  16 &   8 & 384 & 2.80 & (2.5871, 3.0082) & 0.25 & 0.98 \\ 
83 &   1 &  16 &   8 & 281 & 2.94 & (2.6687, 3.2104) & 0.12 & 1.00 \\ 
84 &   1 &   9 &   7 & 354 & 3.36 & (3.0478, 3.6725) & 0.14 & 1.00 \\ 
85 &   1 &  23 &  13 & 575 & 2.53 & (2.3909, 2.6742) & 0.23 & 0.90 \\ 
86 &   1 &  21 &  11 & 616 & 2.80 & (2.6308, 2.9631) & 0.18 & 1.00 \\ 
87 &   1 &  57 &  21 & 1385 & 2.45 & (2.3611, 2.5315) & 0.29 & 0.36 \\ 
88 &   1 &  12 &   9 & 438 & 2.88 & (2.6702, 3.0869) & 0.11 & 1.00 \\ 
89 &   1 &  11 &   9 & 440 & 2.78 & (2.5832, 2.971) & 0.33 & 0.73 \\ 
90 &   1 &  11 &   8 & 348 & 2.78 & (2.5602, 2.9966) & 0.25 & 0.98 \\ 
91 &   1 &  33 &  12 & 710 & 2.52 & (2.3979, 2.6513) & 0.25 & 0.87 \\ 
92 &   1 &  10 &   7 & 367 & 3.09 & (2.83, 3.3523) & 0.29 & 0.96 \\ 
93 &   1 &  26 &  12 & 572 & 2.87 & (2.6892, 3.0519) & 0.08 & 1.00 \\ 
94 &   1 &  13 &   8 & 469 & 3.09 & (2.8594, 3.3212) & 0.12 & 1.00 \\ 
95 &   1 &   4 &   4 & 218 & 3.98 & (3.4266, 4.538) & 0.25 & 1.00 \\ 
96 &   1 &  11 &   9 & 634 & 2.81 & (2.6481, 2.9796) & 0.44 & 0.35 \\ 
97 &   1 &   4 &   4 & 106 & 4.02 & (3.2066, 4.8291) & 0.25 & 1.00 \\ 
98 &   1 &   7 &   7 & 389 & 3.04 & (2.7961, 3.2879) & 0.29 & 0.96 \\ 
99 &   1 &  23 &  15 & 713 & 2.53 & (2.4043, 2.6585) & 0.13 & 1.00 \\ 
100 &   1 &  25 &  16 & 739 & 2.43 & (2.316, 2.5464) & 0.19 & 0.95 \\ 
101 &   1 &  12 &  11 & 554 & 2.60 & (2.4483, 2.7523) & 0.36 & 0.48 \\ 
102 &   1 &  13 &   9 & 575 & 2.94 & (2.7496, 3.1281) & 0.22 & 0.99 \\ 
103 &   1 &  14 &  10 & 312 & 2.69 & (2.4737, 2.9066) & 0.20 & 0.99 \\ 
104 &   1 &  18 &  10 & 535 & 3.10 & (2.8849, 3.3206) & 0.10 & 1.00 \\ 
105 &   1 &  28 &   9 & 525 & 2.77 & (2.5927, 2.9458) & 0.22 & 0.99 \\ 
106 &   1 &  10 &   6 & 341 & 3.11 & (2.8394, 3.389) & 0.33 & 0.93 \\ 
107 &   1 &  16 &   7 & 399 & 3.35 & (3.057, 3.6419) & 0.14 & 1.00 \\ 
108 &   1 &  24 &  15 & 440 & 2.39 & (2.2429, 2.5307) & 0.20 & 0.94 \\ 
109 &   1 &  13 &   9 & 315 & 2.85 & (2.608, 3.0896) & 0.22 & 0.99 \\ 
110 &   1 &  23 &  12 & 690 & 2.58 & (2.448, 2.7167) & 0.25 & 0.87 \\ 
111 &   1 &   5 &   5 & 201 & 3.49 & (3.0435, 3.9358) & 0.20 & 1.00 \\ 
112 &   1 &   8 &   7 & 388 & 3.23 & (2.9498, 3.5012) & 0.29 & 0.96 \\ 
113 &   1 &  16 &  12 & 396 & 2.56 & (2.3888, 2.7385) & 0.17 & 1.00 \\ 
114 &   1 &   4 &   4 & 289 & 3.37 & (3.0246, 3.7212) & 0.25 & 1.00 \\ 
115 &   1 &  12 &   8 & 332 & 3.02 & (2.7618, 3.2885) & 0.25 & 0.98 \\ 
116 &   1 &   7 &   7 & 222 & 3.10 & (2.7606, 3.4349) & 0.29 & 0.96 \\ 
117 &   1 &   7 &   6 & 176 & 2.67 & (2.3872, 2.9559) & 0.50 & 0.47 \\ 
118 &   1 &  15 &   9 & 327 & 2.77 & (2.5436, 2.9903) & 0.22 & 0.99 \\ 
119 &   1 &  27 &  16 & 755 & 2.44 & (2.3285, 2.5588) & 0.31 & 0.43 \\ 
120 &   1 &   5 &   5 &  71 & 2.94 & (2.3979, 3.4724) & 0.40 & 0.87 \\ 
121 &   1 &   8 &   7 & 208 & 2.77 & (2.4876, 3.048) & 0.43 & 0.58 \\ 
122 &   1 &   9 &   4 &  21 & 2.63 & (1.832, 3.4294) & 0.25 & 1.00 \\ 
123 &   1 &  10 &   9 & 416 & 2.67 & (2.4898, 2.8606) & 0.22 & 0.99 \\ 
124 &   1 &  12 &  11 & 438 & 2.79 & (2.5902, 2.9813) & 0.27 & 0.83 \\ 
125 &   1 &  15 &  12 & 378 & 2.63 & (2.4385, 2.8137) & 0.25 & 0.87 \\ 
126 &   1 &  10 &   9 & 421 & 2.65 & (2.4712, 2.8338) & 0.33 & 0.73 \\ 
127 &   1 &  14 &   9 & 221 & 2.61 & (2.3645, 2.8481) & 0.22 & 0.99 \\ 
128 &   1 &  15 &  11 & 450 & 2.62 & (2.4447, 2.7859) & 0.27 & 0.83 \\ 
129 &   1 &  10 &   9 & 197 & 2.54 & (2.2982, 2.7857) & 0.22 & 0.99 \\ 
130 &   1 &  18 &  13 & 561 & 2.63 & (2.4759, 2.7849) & 0.23 & 0.90 \\ 
131 &   1 &   5 &   5 & 188 & 3.16 & (2.7814, 3.544) & 0.40 & 0.87 \\ 
132 &   1 &  14 &  12 & 537 & 2.65 & (2.4893, 2.8096) & 0.17 & 1.00 \\ 
133 &   1 &  29 &  18 & 1023 & 2.40 & (2.3039, 2.4946) & 0.28 & 0.50 \\ 
134 &   1 &  27 &  17 & 997 & 2.48 & (2.3809, 2.5881) & 0.24 & 0.75 \\ 
135 &   1 &  36 &  24 & 1145 & 2.25 & (2.174, 2.3327) & 0.38 & 0.07 \\ 
\end{longtable}
\end{appendices}

%%===========================================================================================%%
%% If you are submitting to one of the Nature Portfolio journals, using the eJP submission   %%
%% system, please include the references within the manuscript file itself. You may do this  %%
%% by copying the reference list from your .bbl file, paste it into the main manuscript .tex %%
%% file, and delete the associated \verb+\bibliography+ commands.                            %%
%%===========================================================================================%%

\bibliographystyle{apalike}
\bibliography{./ref}% common bib file
%% if required, the content of .bbl file can be included here once bbl is generated
%%\input sn-article.bbl

\end{document}